\documentclass[a4paper,11pt]{amsart}

\allowdisplaybreaks[0]
\usepackage{cite}
\usepackage{amsmath}
\usepackage{amssymb}
\usepackage{amsthm}

\numberwithin{equation}{section}
\theoremstyle{plain}
\newtheorem{theorem}{Theorem}[section]
\newtheorem{lemma}[theorem]{Lemma}

\newcommand{\Mtilde}{\widetilde M}
\newcommand{\C}{\mathbb C}
\newcommand{\diver}{\mathop{\rm div}\nolimits}

\title[Hamiltonian stationary Lagrangian surfaces]
{Hamiltonian stationary Lagrangian surfaces of constant Gaussian curvature in complex space forms}

\author{Toru Sasahara}
\address{Division of Mathematics, Center for Liberal Arts and Sciences,
Hachinohe Institute of Technology, Hachinohe, Aomori 031-8501, Japan}
\email{sasahara@hi-tech.ac.jp}
\date{}

\begin{document}

\begin{abstract}
Let $M$ be a connected Hamiltonian stationary Lagrangian surface
of constant Gaussian curvature $K$ in a two-dimensional complex
space form $\Mtilde^2(4\epsilon)$ of constant holomorphic sectional
curvature $4\epsilon$, and let $H$ denote its mean curvature vector field.
We prove that either $K=\epsilon$ or $K=0\ne\epsilon$ and
$|H|$ is constant.  In particular, if $|H|$ is non-constant,
then $K=\epsilon$.
\end{abstract}

\def\subjclassname{\textup{2020} Mathematics Subject Classification}

\maketitle

\section{Introduction}

Let $\Mtilde^n(4\epsilon)$ be a complete, simply connected complex
space form with complex structure $J$, K\"ahler metric
$\langle\cdot,\cdot\rangle$, and holomorphic sectional curvature
$4\epsilon$. An $n$-dimensional submanifold $M$ of
$\Mtilde^n(4\epsilon)$ is Lagrangian if
$\langle X,JY\rangle=0$ for all tangent vector fields $X,Y$ on $M$.
A normal vector field $\xi=J\operatorname{grad}f$, where $f$ is a
smooth compactly supported function on $M$, is called a Hamiltonian
variation. A submanifold $M$ is Hamiltonian stationary if its volume
is critical under all Hamiltonian variations. This condition is
equivalent to
\begin{equation*}
 \diver(JH)=0,
\end{equation*}
where $H$ is the mean curvature vector field and $\diver$ is the
divergence with respect to the induced metric (cf.\ \cite{Oh1993}).

Hamiltonian stationary Lagrangian surfaces in $\Mtilde^2(4\epsilon)$ with constant Gaussian curvature $K=\epsilon$
have been classified
\cite{ChenDillen2008,CGZ2009,Chen2012,CDVV1998}. It was shown
that every Hamiltonian stationary Lagrangian surface with non-zero
constant $|H|$ is flat and has a parallel second fundamental form
\cite{Sasahara2026}. The same paper classified Hamiltonian stationary Lagrangian
surfaces for which $H$ is nowhere vanishing, $K$ is constant, and
$|H|$ is a non-constant harmonic function. For these surfaces,
$K=\epsilon=-1$.

To the best of our knowledge, there are no examples of Hamiltonian
stationary Lagrangian surfaces with constant Gaussian curvature
$K\ne\epsilon$ and non-constant $|H|$. This motivates us to
investigate whether such surfaces can exist.

Our main result is the following.
\begin{theorem}\label{thm:constant-curvature}
Let $M$ be a connected Hamiltonian stationary Lagrangian surface in
$\Mtilde^2(4\epsilon)$ with constant Gaussian curvature $K$. Then either
$K=\epsilon$, or $K=0\ne\epsilon$ and $|H|$ is constant. In particular,
if $|H|$ is non-constant, then $K=\epsilon$.
\end{theorem}

\section{Preliminaries}\label{sec:preliminaries}

We recall the integrability system from
\cite{HMW2019}.

Let $M$ be a Hamiltonian stationary Lagrangian surface in
$\Mtilde^2(4\epsilon)$. Let $g$ and $h$ denote its induced metric and second fundamental form,
respectively. The mean curvature vector is defined by $H=\frac12\operatorname{trace}_g h$.
On an oriented coordinate disc, choose a conformal coordinate $z=x+iy$
such that
\begin{equation*}
 g=2e^u|dz|^2=2e^u(dx^2+dy^2).
\end{equation*}

Define the symmetric cubic tensor
$\mathcal C(X,Y,Z)=\langle h(X,Y),JZ\rangle$. The mean curvature
form $\Phi=\phi\,dz$ is given by
\begin{equation*}
 \Phi+\overline{\Phi}=-\langle JH,\cdot\rangle.
\end{equation*}
Since $M$ is Hamiltonian stationary, $\Phi$ is holomorphic on the coordinate disc (cf.\ \cite{HMW2019}).

Set $\delta=\epsilon-K$ and $\psi=\mathcal C(\partial_z,\partial_z,\partial_z)$. This is the coefficient of the cubic Hopf differential in \cite{HMW2019}.
The relation between $\phi$ and $|H|$ is
\begin{equation}\label{eq:mean-curvature-phi}
 |H|^2=2e^{-u}|\phi|^2.
\end{equation}

Using the above notation, the Gauss and Codazzi equations can be written, respectively, as:
\begin{equation}\label{eq:general-complex-system}
 |\psi|^2=e^{2u}|\phi|^2+\delta e^{3u},
 \qquad \psi_{\bar z}=e^{2u}(e^{-u}\phi)_z.
\end{equation}

\section{A lemma for the flat case}

The following lemma will be used in the proof of
Theorem~\ref{thm:constant-curvature}.

\begin{lemma}\label{lem:flat-holomorphic}
Let $F$ and $G$ be holomorphic functions on a disc $D$ in $\C$.
Suppose that
\begin{equation}\label{eq:flat-holomorphic-relation}
 |\bar z F'(z)+G(z)|^2-|F(z)|^2=\lambda,
\end{equation}
where $\lambda$ is a non-zero real constant.  Then $F$ is constant.
\end{lemma}

\begin{proof}
Assume, to the contrary, that $F$ is non-constant, and set
$v=\bar zF'+G$. Its derivatives are
\begin{equation}\label{eq:flat-v-relation}
 v_{\bar z}=F',\qquad v_z=\bar zF''+G',\qquad v_{z\bar z}=F''.
\end{equation}
By \eqref{eq:flat-holomorphic-relation}, $|v|^2=|F|^2+\lambda$.
Applying $\partial_z\partial_{\bar z}$ to this relation, we obtain
\begin{align*}
 |F'|^2
 &=\partial_z\partial_{\bar z}|F|^2
  =\partial_z\partial_{\bar z}|v|^2\\
 &=2\operatorname{Re}\bigl(v_{z\bar z}\bar v\bigr)
   +|v_z|^2+|v_{\bar z}|^2\\
 &=2\operatorname{Re}(F''\bar v)+|v_z|^2+|F'|^2.
\end{align*}
Hence
\begin{equation}\label{eq:flat-differential-identity}
 2\operatorname{Re}(F''\bar v)+|v_z|^2=0.
\end{equation}

\noindent\textbf{Case (I): $F''\equiv0$.}
Then $F(z)=az+c$ for constants $a\ne0$ and $c$. By
\eqref{eq:flat-differential-identity},
$|v_z|^2=|G'|^2=0$, and hence $G\equiv\tau$ for a constant $\tau$.
Equation~\eqref{eq:flat-holomorphic-relation} becomes
\begin{equation}\label{eq:flat-linear-case}
 |\tau|^2-|c|^2+2\operatorname{Re}\bigl((\bar a\tau-a\bar c)z\bigr)
 =\lambda.
\end{equation}
The real part of the holomorphic function
$(\bar a\tau-a\bar c)z$ is therefore constant, and the
Cauchy--Riemann equations give $\bar a\tau-a\bar c=0$.
Since $a\ne0$, it follows that $|\tau|=|c|$, and hence $\lambda=0$,
a contradiction. Thus Case (I) is impossible.

\smallskip
\noindent\textbf{Case (II): $F''\not\equiv0$.}
Then $F'F''\not\equiv0$ by the identity theorem. Choose a disc
$D_0\subset D$ on which $F'F''\ne0$.
From here on, we work on $D_0$.
\begin{equation}\label{eq:flat-auxiliary-functions}
 \alpha=\frac{F'}{F''},\qquad
 \beta=\frac{G'}{F''},\qquad
 \gamma=\frac{G}{F''}.
\end{equation}
By these definitions,
\begin{equation}\label{eq:flat-v-factorization}
 v=F''(\bar z\alpha+\gamma),\qquad
 v_z=F''(\bar z+\beta).
\end{equation}
Substituting \eqref{eq:flat-v-factorization} into
\eqref{eq:flat-differential-identity} and dividing by $|F''|^2$,
we obtain
\begin{equation}\label{eq:flat-local-identity}
 2\operatorname{Re}(\bar z\alpha+\gamma)+|\bar z+\beta|^2=0.
\end{equation}
Applying $\partial_z\partial_{\bar z}$ to this equation gives
\begin{equation}\label{eq:flat-alpha-beta-identity}
 2\operatorname{Re}\alpha'+1+|\beta'|^2=0.
\end{equation}
Since $\operatorname{Re}\alpha'$ is harmonic, applying
$\partial_z\partial_{\bar z}$ once more shows that $|\beta''|^2=0$.
It follows that $\beta$ has the form
\begin{equation}\label{eq:flat-beta-affine}
 \beta=\beta_0+\beta_1z.
\end{equation}
By \eqref{eq:flat-alpha-beta-identity},
$\operatorname{Re}\alpha'$ is constant, which implies that $\alpha'$ is
constant. Writing $\alpha'=\alpha_1$, we have
\begin{equation}\label{eq:flat-alpha-affine}
 \alpha=\alpha_0+\alpha_1z,\qquad
 \operatorname{Re}\alpha_1=-\frac{1+|\beta_1|^2}{2}<0.
\end{equation}

Applying $\partial_z^2$ to \eqref{eq:flat-local-identity}, we obtain
\begin{equation}\label{eq:flat-gamma-second-derivative}
 \gamma''+2\beta_1=0.
\end{equation}
Integrating \eqref{eq:flat-gamma-second-derivative} twice, we obtain
\begin{equation}\label{eq:flat-gamma-form}
 \gamma=\gamma_0+\gamma_1z-\beta_1z^2
\end{equation}
for constants $\gamma_0$ and $\gamma_1$.

Differentiating $F'=\alpha F''$ and $G=\gamma F''$, and using
$\alpha'=\alpha_1$ and $G'=\beta F''$, we obtain
\begin{equation*}
 \alpha F'''=(1-\alpha_1)F'',\qquad
 \beta F''=\gamma'F''+\gamma F'''.
\end{equation*}
Eliminating $F'''$ gives
\begin{equation}\label{eq:flat-polynomial-identity}
 (1-\alpha_1)\gamma=\alpha(\beta-\gamma').
\end{equation}
Substituting \eqref{eq:flat-beta-affine},
\eqref{eq:flat-alpha-affine}, and \eqref{eq:flat-gamma-form} into
\eqref{eq:flat-polynomial-identity}, we get
\begin{equation}\label{eq:flat-polynomial-coefficients}
 (1-\alpha_1)(\gamma_0+\gamma_1z-\beta_1z^2)
 =(\alpha_0+\alpha_1z)(\beta_0-\gamma_1+3\beta_1z).
\end{equation}
Comparing the coefficients of $z^2$ in
\eqref{eq:flat-polynomial-coefficients}, we find
\begin{equation}\label{eq:flat-beta1-condition}
 -(1-\alpha_1)\beta_1=3\alpha_1\beta_1,
 \qquad\text{hence}\qquad \beta_1(1+2\alpha_1)=0.
\end{equation}
If $\beta_1\ne0$, \eqref{eq:flat-beta1-condition} forces
$\alpha_1=-1/2$, which, by \eqref{eq:flat-alpha-affine}, implies
$|\beta_1|^2=0$, a contradiction. Hence $\beta_1=0$.

Comparing the coefficients of $z$ and $1$ in
\eqref{eq:flat-polynomial-coefficients} yields
\begin{equation}\label{eq:flat-constant-coefficients}
 (1-\alpha_1)\gamma_1=\alpha_1(\beta_0-\gamma_1),\qquad
 (1-\alpha_1)\gamma_0=\alpha_0(\beta_0-\gamma_1).
\end{equation}
Solving \eqref{eq:flat-constant-coefficients}, we obtain
$\gamma_1=\alpha_1\beta_0$ and $\gamma_0=\alpha_0\beta_0$, the latter
since $1-\alpha_1\ne0$ by $\operatorname{Re}\alpha_1<0$. Hence
$\gamma=\beta_0\alpha$, that is, $G=\beta_0F'$ on $D_0$.

Set $P(z):=(z+\bar\beta_0)F'(z)$, a holomorphic function on
$D_0$. Since $v=(\bar z+\beta_0)F'(z)$, we have $|v|=|P|$, and
\eqref{eq:flat-holomorphic-relation} becomes
\begin{equation}\label{eq:flat-P-relation}
 |P|^2-|F|^2=\lambda.
\end{equation}
Applying $\partial_z\partial_{\bar z}$ to this equation yields
\begin{equation}\label{eq:flat-P-derivative-identity}
 |P'|^2=|F'|^2.
\end{equation}
Since $F'\ne0$ on $D_0$, $P'/F'$ is holomorphic on $D_0$
and has modulus one by \eqref{eq:flat-P-derivative-identity}. Hence it
is a constant $\mu$ with $|\mu|=1$, by the maximum modulus principle.
Consequently, $P=\mu F+b$ for a constant $b$. Substituting this into
\eqref{eq:flat-P-relation} gives
\begin{equation}\label{eq:flat-real-part-identity}
 2\operatorname{Re}(\mu\bar b F)+|b|^2=\lambda.
\end{equation}
If $b\ne0$, \eqref{eq:flat-real-part-identity} shows that the holomorphic
function $\mu\bar bF$ has constant real part. The Cauchy--Riemann equations
then imply that $\mu\bar bF$ is constant. Since $\mu\bar b\ne0$, this
forces $F'=0$ on $D_0$, contradicting the choice of $D_0$. If $b=0$,
\eqref{eq:flat-real-part-identity} forces $\lambda=0$, also a contradiction.

The contradictions in Cases (I) and (II) rule out the assumption that
$F$ is non-constant. Therefore, $F$ is constant.
\end{proof}

\section{Proof of the main theorem}

\begin{proof}[Proof of Theorem~\ref{thm:constant-curvature}]
We use the notation of
Section~\ref{sec:preliminaries}.
We divide the proof into two cases.

\smallskip
\noindent\textbf{Case (I): $K=0$.}
If $\epsilon=0$, then $K=\epsilon$.
Suppose that $\epsilon\ne0$ and choose local Euclidean coordinates with
$g=|dz|^2$, which gives $e^u=1/2$. By the Codazzi equation in
\eqref{eq:general-complex-system},
\begin{equation*}
 \psi_{\bar z}=\frac12\phi'.
\end{equation*}
Integrating with respect to $\bar z$, we obtain
\begin{equation*}
 \psi=\frac12\bar z\phi'+G,
\end{equation*}
where $G$ is holomorphic. The Gauss equation in \eqref{eq:general-complex-system} takes the form
\begin{equation*}
 \left|\frac12\bar z\phi'(z)+G(z)\right|^2
 -\frac14|\phi(z)|^2=\frac{\epsilon}{8}.
\end{equation*}
This relation has the form in Lemma~\ref{lem:flat-holomorphic} with
$F=\phi/2$ and $\lambda=\epsilon/8$. Since $\epsilon\ne0$, the lemma
implies that $\phi$ is locally constant.
Equation~\eqref{eq:mean-curvature-phi} then shows that $|H|$ is locally
constant, and connectedness then makes $|H|$ constant on $M$.

\smallskip
\noindent\textbf{Case (II): $K\ne0$.}
Set $\kappa=K/4$. Since $K=4\kappa$ is constant, $(M,g)$ is locally isometric to the surface with metric $w^{-2}|dz|^2$, where $w=1+\kappa|z|^2$. Hence around any point there is an oriented conformal coordinate $z$, centered at that point, such that $g=w^{-2}|dz|^2$. Choose $r>0$ satisfying $|\kappa|r^2<1$ and let $\Delta=\{z:|z|<r\}$. Then $w=1+\kappa|z|^2>0$ on $\Delta$ and $e^u=1/(2w^2)$.
The Codazzi equation in \eqref{eq:general-complex-system} becomes
\begin{equation*}
 \psi_{\bar z}=\frac{\phi'}{2w^2}
                 +\frac{\kappa\bar z\phi}{w^3}.
\end{equation*}
Integration with respect to $\bar z$ yields
\begin{equation}\label{eq:general-integrated-cubic}
 \psi=G(z)+\frac{\bar z\phi'(z)}{2w}
       +\frac{\kappa\bar z^2\phi(z)}{2w^2},
\end{equation}
where $G$ is holomorphic.

Define the holomorphic functions by
\begin{equation*}
 A_0=G,\qquad A_1=2\kappa zG+\frac12\phi',\qquad
 A_2=\kappa^2z^2G+\frac\kappa2z\phi'+\frac\kappa2\phi.
\end{equation*}
Then we have
\begin{equation}\label{eq:cubic-polynomial-expansion}
 w^2\psi=A_0(z)+A_1(z)\bar z+A_2(z)\bar z^2.
\end{equation}
The Gauss equation in \eqref{eq:general-complex-system} yields
\begin{equation}\label{eq:general-finite-rank}
 w^4|\psi|^2-\frac14|\phi|^2=\frac{\delta}{8(1+\kappa|z|^2)^2}.
\end{equation}

Taking the squared modulus of \eqref{eq:cubic-polynomial-expansion}, we obtain
\begin{equation*}
 w^4|\psi|^2
 =\sum_{i,j=0}^2
   [z^jA_i(z)][\bar z^i\overline{A_j(z)}].
\end{equation*}
Since $\Delta$ is invariant under complex conjugation, for a holomorphic function $L$ on $\Delta$ define
$L^*(\zeta)=\overline{L(\bar\zeta)}$ for $\zeta\in\Delta$.
This is holomorphic. Replacing $\bar z$ by an independent variable
$\zeta$ and each conjugate holomorphic function by its reflected function
in \eqref{eq:general-finite-rank}, both sides are holomorphic in
$(z,\zeta)$ and coincide for $\zeta=\bar z$. Hence they coincide identically:
\begin{equation}\label{eq:general-polarized-rank}
 \sum_{i,j=0}^2
 [z^jA_i(z)][\zeta^iA_j^*(\zeta)]
 -\frac14\phi(z)\phi^*(\zeta)
 =\frac{\delta}{8(1+\kappa z\zeta)^2}.
\end{equation}

The definitions of $A_0,A_1,A_2$ give
\begin{equation*}
 \phi=\frac{2}{\kappa}\bigl(A_2-\kappa zA_1+\kappa^2z^2A_0\bigr),
\end{equation*}
and hence $\phi$ belongs to the span of the functions $z^jA_i$. Let
\begin{equation*}
 \mathcal V=\operatorname{span}_{\C}
 \{z^jA_i:0\le i,j\le2\}.
\end{equation*}
Then $\dim_{\C}\mathcal V\le9$.
Since $\phi\in\mathcal V$, differentiating \eqref{eq:general-polarized-rank} $n$ times with respect to $\zeta$ and evaluating at $\zeta=0$ shows that

\begin{equation*}
 \frac{\delta}{8}(-\kappa)^n(n+1)!z^n\in\mathcal V,
 \qquad n=0,1,\ldots,9.
\end{equation*}
If $\delta\ne0$, all ten monomials $1,z,\ldots,z^9$ belong
to $\mathcal V$, since $\kappa\ne0$. Their linear independence contradicts
$\dim_{\C}\mathcal V\le9$.
Hence $\delta=0$, that is, $K=\epsilon$.
\end{proof}

\medskip
\noindent\textbf{AI Disclosure.}\quad
The author used ChatGPT (OpenAI) to assist in exploring proof strategies.
The author verified all mathematical arguments and takes full
responsibility for the content of this manuscript.


\begin{thebibliography}{9}

\bibitem{ChenDillen2008}
B.-Y. Chen and F. Dillen,
Warped product decompositions of real space forms and
Hamiltonian-stationary Lagrangian submanifolds,
\emph{Nonlinear Anal.} \textbf{69} (2008), no.~10, 3462--3494.

\bibitem{CDVV1998}
B.-Y. Chen, F. Dillen, L. Verstraelen and L. Vrancken,
Lagrangian isometric immersions of a real-space-form $M^n(c)$
into a complex-space-form $\Mtilde^n(4c)$,
\emph{Math. Proc. Cambridge Philos. Soc.}
\textbf{124} (1998), 107--125.

\bibitem{CGZ2009}
B.-Y. Chen, O. J. Garay and Z. Zhou,
Hamiltonian-stationary Lagrangian surfaces of constant curvature
$\epsilon$ in complex space form $\Mtilde^2(4\epsilon)$,
\emph{Nonlinear Anal.} \textbf{71} (2009), no.~7--8, 2640--2659.

\bibitem{Chen2012}
B.-Y. Chen,
Solutions to over-determined systems of partial differential equations
related to Hamiltonian stationary Lagrangian surfaces,
\emph{Electron. J. Differential Equations} (2012), no.~83, 1--7.

\bibitem{HMW2019}
H. X. He, H. Ma and E. X. Wang,
Lagrangian Bonnet problems in complex space forms,
\emph{Acta Math. Sin. (Engl. Ser.)} \textbf{35} (2019), no.~8,
1357--1366.

\bibitem{Oh1993}
Y. G. Oh,
Volume minimization of Lagrangian submanifolds under Hamiltonian
deformations,
\emph{Math. Z.} \textbf{212} (1993), 175--192.

\bibitem{Sasahara2026}
T. Sasahara,
Hamiltonian stationary Lagrangian surfaces with harmonic mean
curvature in complex space forms,
\emph{Mediterr. J. Math.} \textbf{23} (2026), no.~5, Paper No.~151.

\end{thebibliography}
\end{document}